\documentclass[a4paper,11pt]{amsart}

\usepackage{amsmath}
\usepackage{eucal}
\usepackage{amssymb}
\usepackage{xy}
\usepackage{cite}

\newtheorem*{theorem}{Theorem}

\newtheorem*{corollary}{Corollary}
\newtheorem*{lemma}{Lemma}
\newtheorem*{conjecture}{Conjecture}
\theoremstyle{definition}

\theoremstyle{remark}

\begin{document}
\bibliographystyle{abbrv}

\title[Green-Griffiths locus of projective hypersurfaces]{A remark on the codimension of the Green-Griffiths locus of generic projective hypersurfaces of high degree}

\author{Simone Diverio \and Stefano Trapani}
\address{Simone Diverio \\ Institut de Math\'ematiques de Jussieu \\ Universit\'e Pierre et Marie Curie, Paris}
\email{diverio@math.jussieu.fr} 
\address{Stefano Trapani \\ Dipartimento di matematica \\ Universit\`a di Roma \lq\lq Tor Vergata\rq\rq{}}
\email{trapani@mat.uniroma2.it}

\keywords{Invariant jet differential, slanted vector field, Green-Griffiths locus, projective hypersurface, Kobayashi's conjecture}
\subjclass{Primary: 32H20, 32J25; Secondary: 14F05, 14J70}

\begin{abstract}
We show that for every smooth generic projective hypersurface $X\subset\mathbb P^{n+1}$, there exists a proper subvariety $Y\subsetneq X$ such that $\operatorname{codim}_X Y\ge 2$ and for every non constant holomorphic entire map $f\colon\mathbb C\to X$ one has $f(\mathbb C)\subset Y$, provided $\deg X\ge 2^{n^5}$. In particular, we obtain an effective confirmation of the Kobayashi conjecture for threefolds in $\mathbb P^4$.
\end{abstract}

\maketitle

\section{introduction}

Let $X$ be a compact complex manifold endowed with an ample line bundle $A\to X$ and consider the $k$-th projectivized jet bundle $\pi_{0,k}\colon X_k\to X$ as introduced in \cite{Dem97}. Denote $\mathcal O_{X_k}(m)\to X_k$ the $m$-th power of the antitautological line bundle and consider all the base loci of their twisting with $\pi_{0,k}^*\mathcal O(-A)$:  call $B_k$ their intersection as $m>0$ vary.

Following \cite{Dem97}, we call 
$$
GG(X)=\bigcap_{k>0}\pi_{0,k}(B_k)\subset X
$$ 
the \emph{Green-Griffiths locus} of $X$. It is well known \cite{GG80, Dem97, SY97}, that this locus must contain the image of any nonconstant holomorphic entire map from $\mathbb C$ to $X$.

In the recent paper \cite{DMR}, the Green-Griffiths locus is studied for $X$ a generic smooth projective hypersurface of high (effective, grater than or equal to $2^{n^5}$, where $\dim(X)=n$) degree and it is shown that this is a proper subvariety. 

In the same setting, we exclude here the possibility for $GG(X)$ to have divisorial components.

\begin{theorem}\label{main}
Let $X\subset\mathbb P^{n+1}$ be a generic smooth projective hypersurface of degree $\deg X\ge 2^{n^5}$. Then, the Green-Griffiths locus $GG(X)\subsetneq X$ is a proper subvariety of codimension at least two.
\end{theorem}

The proof goes along the same lines of \cite{DMR}, but with a slightly different point of view. We recall here briefly that to show the positivity of the codimension of the Green-Griffiths locus, \cite{DMR} use a first invariant jet differential of order $n$ and high weighted degree ---the existence of which is shown in \cite{Div09}--- and then, following the strategy indicated in \cite{Siu04}, by differentiating with the meromorphic vector fields constructed in \cite{Mer09}, they are able to produce new, algebraically independent jet differentials in order to show the aforesaid positivity of the codimension of the Green-Griffiths locus. \emph{A posteriori}, this locus $Y$ coincides with the zero locus of the first jet differential, seen as a section of a vector bundle on $X$, and nothing more is said in \cite{DMR} since there is no control of its singularity in order to bound the number of differentiation needed to reduce it.
Here, thanks to a general and very simple remark on the zero locus of holomorphic sections of vector bundles, we show that to be able to exclude divisorial components in the zero locus, it is not crucial to know its singularities (see Section \ref{proofTheorem} for the details). However we feel that this slightly different approach is specific for the codimension one, and probably a substantial refinement is needed to increase more the codimension (a straightforward blow-up approach, for example, seems to be not enough).

The Kobayashi conjecture \cite{Kob70}, states that every generic projective hypersurface $X\subset\mathbb P^{n+1}$ is Kobayashi hyperbolic provided $\deg X\ge 2n+1$. Thanks to a classical lemma due to Brody, this is equivalent for $X$ to admit only holomorphic constant maps from $\mathbb C$. A solution to this conjecture has been proposed in \cite{Siu04}, while the conjecture is solved in \cite{McQ99,DEG00,Pau08} for the case of (very) generic surface in projective $3$-space with effective bounds given respectively by $\deg X\ge 36, 21, 18$. As a consequence of the above theorem, we obtain ---thanks to a result contained in \cite{Cle86}--- an \emph{effective} confirmation of this conjecture for threefolds in projective $4$-space.

\begin{corollary}
Let $X\subset\mathbb P^4$ be a (very) generic smooth hypersurface. If $\deg X\ge 593$, then $X$ is Kobayashi hyperbolic.
\end{corollary}

This result is an improvement of the corresponding statement contained in \cite{Rou07} where, with the same effective bound, it is shown that entire curves cannot be Zariski dense.

\bigskip

For survey material, notations and background we refer all along this paper to \cite{GG80,Dem97,Siu04}.

\section{Proof of the theorem}\label{proofTheorem}

Let $E_{k,m}T^*_X\to X$ be the vector bundle of invariant jet differentials of order $k$ and weighted degree $m$ over the complex manifold $X$. It is shown in \cite{Div09,DMR} that, whenever $X\subset\mathbb P^{n+1}$ is a smooth projective hypersurface of large degree, for every $q>0$, the space of global section $H^0(X,E_{n,m}T^*_X\otimes\mathcal O_X(-mq))$ is nonzero, provided $\deg X$ and $m$ are large enough. This is the starting point of algebraic degeneracy in \cite{DMR}.

\subsection{The proof of algebraic degeneracy of \cite{DMR}}

Here, we outline how the proof of algebraic degeneracy of \cite{DMR} works. 

Start with a nonzero section $P\in H^0(X,E_{n,m}T^*_X\otimes\mathcal O_X(-mq))$, for some $m\gg 0$, where $X\subset\mathbb P^{n+1}$ is a smooth generic projective hypersurface of degree $d$ large enough (in order to have such a section).
Call 
$$
Y=\{P=0\}\subsetneq X
$$ 
the zero locus of such a nonzero section. Look at $P$ as an invariant (under the action of the group $\mathbb G_n$ of $n$-jets of biholomorphic changes of parameter of $(\mathbb C,0)$) map 
$$
J_nT_X\to p^*\mathcal O_X(-mq)
$$
where $p\colon J_n T_X\to X$ is the space of $n$-jets of germs of holomorphic curves $f\colon (\mathbb C,0)\to X$. Then $P$ is a weighted homogeneous polynomial in the jet variables of degree $m$ with coefficients holomorphic functions of the coordinates of $X$ and values in $p^*\mathcal O_X(-mq)$.

Suppose for a moment that we have enough global holomorphic $\mathbb G_n$-invariant vector fields on $J_n T_X$ with values in the pull-back from X of some ample divisor in order to generate $T_{J_n T_X}\otimes p^*\mathcal O_X(\ell)$, at least over the dense open set $J_n^{\text{\rm reg}}T_X$ of regular $n$-jets, \emph{i.e.} of $n$-jets with nonvanishing first derivative. 

If $f\colon\mathbb C\to X$ is an entire curve, consider its lifting $j_n(f)\colon\mathbb C\to J_n T_X$ and suppose that $j_n(f)(\mathbb C)\subsetneq J_n^{\text{\rm sing}}T_X\overset{\text{\rm def}}=J_n T_X\setminus J_n^{\text{\rm reg}}T_X$ (otherwise $f$ is constant). Arguing by contradiction, let $f(\mathbb C)\subsetneq Y$ and $x_0=f(t_0)\in X\setminus Y$.
Thus, one can produce, by differentiating at most $m$ times, a new invariant $n$-jet differential $Q$ of weighted degree $m$ with values in $\mathcal O_X(m\ell-mq)$ such that $Q(j_n(f)(t_0))\ne 0$, thus contradicting the well-known results based on the Ahlfors-Schwartz lemma contained in \cite{GG80,Dem97} (see also \cite{SY97}), provided $q>\ell$, \emph{i.e.} provided $Q$ is still with value in an antiample divisor.

Unfortunately, in general we don't know if one can hope for such a global generation statement for meromorphic vector fields of $J_nT_X$ which would bring to a confirmation of the Green-Griffiths conjecture for \emph{all} smooth projective hypersurfaces of high degree and not only for the generic one. Thus, as in \cite{Siu04,Pau08,Rou07,DMR}, one has to use \lq\lq slanted vector fields\rq\rq{} in order to gain some positivity.

Consider the universal hypersurface $\mathcal X\subset\mathbb P^{n+1}\times\mathbb P(H^0(\mathbb P^{n+1},\mathcal O(d)))$ of degree $d$ in $\mathbb P^{n+1}$. Next, consider the subbundle $\mathcal V\subset T_{\mathcal X}$ given by the kernel of the differential of the second projection. If $s\in\mathbb P(H^0(\mathbb P^{n+1},\mathcal O(d)))$ parametrizes any smooth hypersurface $X_s$, then one has
$$
H^0(X_s,E_{n,m}T^*_{X_s}\otimes\mathcal O_{X_s}(-mq))\simeq H^0(X_s,E_{n,m}\mathcal V^*\otimes\text{pr}_1^*\mathcal O(-mq)|_{X_s}).
$$
Suppose $X=X_0$ corresponds to the parameter $0\in\mathbb P(H^0(\mathbb P^{n+1},\mathcal O(d)))$. Since we have chosen $X$ to be generic, standard semicontinuity arguments show that there exists an open neighborhood $U\ni 0$ such that the restriction morphism
$$
H^0(\text{pr}_2^{-1}(U),E_{n,m}\mathcal V^*\otimes\text{\rm pr}_1^*\mathcal O(-mq))\to H^0(X_0,E_{n,m}T^*_{X_0}\otimes\mathcal O_{X_0}(-mq))
$$
is surjective. Therefore the \lq\lq first\rq\rq{} jet differential may be extended to a neighborhood of the starting hypersurface, and one can use the following global generation statement.

\begin{theorem}[\cite{Mer09}, compare also with \cite{Siu04}]\label{gg}
The twisted tangent bundle 
$$
T_{J_n\mathcal V}\otimes\text{\rm pr}_1^*\mathcal O(n^2+2n)\otimes\text{\rm pr}_2^*\mathcal O(1)
$$
is generated over $J_n^{\text{\rm reg}}\mathcal V$ by its global sections. Moreover, one can choose such generating global sections to be invariant under the action of $\mathbb G_n$ on $J_n\mathcal V$. 
\end{theorem}
Thus, by replacing $\ell$ by $n^2+2n$ in our previous discussion and by removing the hyperplane which corresponds to the poles given by $\text{\rm pr}_2^*\mathcal O(1)$ in the parameter space, one gets the desired result of algebraic degeneracy.

Observe that no information is known about the multiplicity of the subvariety $Y$, thus we are not able to bound \emph{a priori} the number of derivative needed in order to reduce the vanishing locus of the first jet differential.

Finally, remark that as a byproduct of the proof we obtain also that the degeneracy proper subvarieties $Y$ form in fact a family and deform together with the hypersurfaces.

\subsection{Remark on the codimension of the Green-Griffiths locus}

The starting point is the following general, straightforward remark. Let $E\to X$ be a holomorphic vector bundle over a compact complex manifold $X$ and let $\sigma\in H^0(X,E)\ne 0$; then, up to twisting by the dual of an effective divisor, one can suppose that the zero locus of $\sigma$ has no divisorial components. This is easily seen, for let $D$ be the divisorial (and effective) part of the zero locus of $\sigma$ and twist $E$ by $\mathcal O_X(-D)$. Then, $\sigma$ is also a holomorphic section of $H^0(X,E\otimes\mathcal O_X(-D))$ and seen as a section of this new bundle, it vanishes on no codimension $1$ subvariety of $X$.

Now, we use this simple remark in our case, the vector bundle $E$ being here $E_{n,m}T^*_X\otimes\mathcal O_X(-mq)$. If $\dim X=2$ and $X$ is very generic or $\dim X\ge 3$ and $X$ is smooth, then the Picard group $\operatorname{Pic}(X)$ of $X$ is equal to $\mathbb Z$. Thus, in these cases, the corresponding $\mathcal O_X(-D)$ is antiample and, in fact, it is $\mathcal O_X(-h)$ for some positive integral $h$. After all, we have shown that in the prove of \cite{DMR}, one can suppose that the \lq\lq first\rq\rq{} invariant jet differential vanishes at most on a codimension $2$ subvariety of $X$, provided one looks at it as a section of $H^0(E_{n,m}T^*_X\otimes\mathcal O_X(-mq-h))$.

The condition $q>\ell$ which has to be fulfilled in order to use the meromorphic vector fields with low pole order is replaced now by $q+h/m>\ell$, which is obviously still satisfied. Thus, nothing change in the effective estimates of \cite{DMR}, and they apply directly even after this remark.

Observe, finally, that even if the antiample line bundle which is used in \cite{DMR} to twist the bundle of invariant jet differentials was $K_X^{-\delta m}$ for some small rational $\delta>0$, instead of $\mathcal O_X(-mq)$, this does not affects the reasoning nor the estimates, the Picard group of $X$ being $\mathbb Z$. Thus, after all, the use of $\mathcal O_X(-mq)$ is just a notational difference.

\section{Proof of the corollary}

Let $X\subset\mathbb P^4$ be a smooth generic $3$-fold in the projective $4$-space.
To get the considerably better lower bound $\deg X\ge 593$, we adopt here the same methods  contained in \cite{Rou07}, together with the study of the algebra of Demailly-Semple invariants of \cite{Rou06a,Rou06b} in order to obtain global section of invariant $3$-jet differentials. 

We recall that in \cite{Rou07}, the result about non-denseness of entire curves is obtained by the following dichotomy: either the non constant entire curve has its first three derivatives linearly dependent or it has to be contained in the zero locus of a section of the bundle of invariant $3$-jet differentials. This dichotomy comes from a weaker form of the global generation statement of Theorem \ref{gg} contained in \cite{Rou07}, which gives better pole order and suffices for our purposes in dimension $3$. 

In the first case, it is stated there that then the entire curve must lie in a hyperplane section of the $3$-fold, while in the second case it is obviously algebraic degenerate.

\begin{lemma}
Let $f\colon\mathbb C\to\mathbb C^{N}$ be a holomorphic map. If $f'\wedge f''\wedge\cdots\wedge f^{(k)}\equiv 0$, then $f(\mathbb C)$ lies inside a codimension $N-k+1$ affine linear subspace. 
\end{lemma}

\begin{proof}
Without loss of generality, we can suppose $k>1$, $f'\wedge f''\wedge\cdots\wedge f^{(k-1)}\not\equiv 0$, $f'(0)\ne 0$ and $(f'\wedge f''\wedge\cdots\wedge f^{(k-1)})(0)\ne 0$. Then there exists an open neighborhood $\Omega\subset\mathbb C$ of $0$ such that for each $t\in\Omega$ we have a linear combination
$$
f^{(k)}(t)=\sum_{j=1}^{k-1}\lambda_{j}(t)\,f^{(j)}(t)
$$
and the $\lambda_j$'s depend holomorphically on $t$. By taking derivatives, one sees inductively that, in $\Omega$, every $f^{(\ell)}$, $\ell\ge k$, is a linear combination of the $f^{(j)}$'s, $1\le j\le k-1$. Thus, all the derivatives in $0$ of $f$ lie in the linear space generated by $f'(0),\dots,f^{(k-1)}(0)$. The conclusion follows by expanding $f$ in power series at $0$.
\end{proof}

This lemma shows that, if we are in the first case, then in fact the image of the entire curve lies in a codimension two subvariety of $X$ (the intersection of $X$ with a codimension two linear subspace of $\mathbb P^4$), provided $X$ is generic.

By the result of the previous section, even in the second case the image of the entire curve must lie in a codimension two subvariety of $X$. 

This means that the Zariski closure of the image of a nonconstant entire curve, if any, must be an algebraic curve in $X$. Then, such an algebraic curve must be rational or elliptic. But this contradicts the following classical result by Clemens, provided $X$ is very generic:

\begin{theorem}[\cite{Cle86}]
Let $X\subset\mathbb P^{n+1}$ be a smooth (very) generic hypersurface. Then $X$ contains no rational curves (resp. elliptic curves) provided $\deg X\ge 2n$ (resp. $2n+1$).
\end{theorem} 

\section{A question about the positivity of $E_{k,m}T^*_X$}

In \cite{Deb05}, O. Debarre  made the following

\begin{conjecture}
The cotangent bundle of the intersection in $\mathbb P^n$ of at least $n/2$ general hypersurfaces of sufficiently high degree is ample.
\end{conjecture}

It is well known that the ampleness of the cotangent bundle implies the hyperbolicity of the manifold. In this sense, this conjecture can also be seen as a conjecture on hyperbolicity of generic complete intersection of high degree and codimension.

More generally, one can look at the bundles of invariant jet differentials $E_{k,m}T^*_X$ as a possible generalization to higher order of symmetric differentials: recall that one has in fact that $E_{1,m}T^*_X= S^m T^*_X$. It is shown in \cite{Dem97} that if $E_{k,m}T^*_X\to X$ is ample for some $k$, then $X$ is Kobayashi hyperbolic.

On the other hand, by a result obtained by the first named author in \cite{Div08}, one has the vanishing
$$
H^0(X,E_{k,m}T^*_X)=0,\quad 0<k<\dim(X)/\operatorname{codim}(X), 
$$
for $X\subset\mathbb P^n$ a smooth complete intersection.

Thus, the conjecture of O. Debarre says in particular that, for $X$ a complete intersection of high degree, as soon as the codimension becomes big enough to avoid the vanishing theorem above, one immediately has ampleness, provided $X$ is sufficiently generic.

In the same vein, it is tempting to propose here the following generalization:

\begin{conjecture}
Let $X\subset\mathbb P^n$ be the intersection of at least $n/(k+1)$ (very) general hypersurfaces of sufficiently high degree. Then, $E_{k,m}T^*_X\to X$ is ample and therefore $X$ is hyperbolic.
\end{conjecture}

We would like to conclude by mentioning that for the case $k=1$, that is the original one by O. Debarre, it has been verified by D. Brotbek \cite{Bro10} ---as an evidence toward the conjecture--- that all polynomials in the Chern classes of the cotangent bundle of such complete intersections that are positive whenever the bundle is ample (see \cite{FL83}), are in fact positive.

\end{document}